\documentclass{bmcart}

\usepackage{amsfonts}
\usepackage{amsfonts}
\usepackage{amsfonts}
\usepackage{amsfonts}
\usepackage{pdflscape}
\usepackage[hidelinks]{hyperref}
\usepackage{multicol}

\newcommand{\ic}[1]{\left<#1\right>}
\newcommand{\fl}[1]{\left\lfloor#1\right\rfloor}

\newcommand{\p}[1]{(#1)}
\newcommand{\pb}[1]{\left(#1\right)}

\usepackage{amsthm}
\usepackage{amsmath,amssymb}

\usepackage{adjustbox}
\usepackage{graphicx}

\allowdisplaybreaks[4]

\numberwithin{equation}{section}

\newtheorem{thm}{Theorem}
\newtheorem{pro}{Proposition}

\newtheorem{con}{Conjecture}
\newtheorem{definition}{Definition}

\usepackage[utf8]{inputenc} %unicode support
\startlocaldefs
\endlocaldefs

\begin{document}

%%% Start of article front matter
\begin{frontmatter}

\begin{fmbox}
\dochead{Research}


\title{On Some Solutions to Hofstadter's $V$-Recurrence}

%%%%%%%%%%%%%%%%%%%%%%%%%%%%%%%%%%%%%%%%%%%%%%
%%                                          %%
%% Enter the authors here                   %%
%%                                          %%
%% Specify information, if available,       %%
%% in the form:                             %%
%%   <key>={<id1>,<id2>}                    %%
%%   <key>=                                 %%
%% Comment or delete the keys which are     %%
%% not used. Repeat \author command as much %%
%% as required.                             %%
%%                                          %%
%%%%%%%%%%%%%%%%%%%%%%%%%%%%%%%%%%%%%%%%%%%%%%

\author[
   addressref={aff1},                  % id's of addresses, e.g. {aff1,aff2}
   corref={aff1},                       % id of corresponding address, if any
 %  noteref={n1},                        % id's of article notes, if any
   email={altug.alkan@pru.edu.tr}   % email address
]{\inits{AA}\fnm{Altug} \snm{Alkan}}
\author[
   addressref={aff2},
   email={nfox@wooster.edu}
]{\inits{}\fnm{Nathan} \snm{Fox}}
\author[
   addressref={aff1,aff3},
   email={}
]{\inits{}\fnm{Orhan Ozgur} \snm{Aybar}}
\author[
   addressref={aff4},
   email={}
]{\inits{}\fnm{Zehra} \snm{Akdeniz}}

%%%%%%%%%%%%%%%%%%%%%%%%%%%%%%%%%%%%%%%%%%%%%%
%%                                          %%
%% Enter the authors' addresses here        %%
%%                                          %%
%% Repeat \address commands as much as      %%
%% required.                                %%
%%                                          %%
%%%%%%%%%%%%%%%%%%%%%%%%%%%%%%%%%%%%%%%%%%%%%%

\address[id=aff1]{%                           % unique id
  \orgname{Graduate School of Science and Engineering, Piri Reis University}, % university, etc
  \street{Tuzla},                     %
  %\postcode{}                                % post or zip code
  \city{Istanbul},                              % city
  \cny{Turkey}                                    % country
}
\address[id=aff2]{%                           % unique id
  \orgname{Department of Mathematical and Computational Sciences, The College of Wooster}, % university, etc
  %\street{Tuzla},                     %
  %\postcode{}                                % post or zip code
  \city{Wooster},                              % city
  \state{Ohio},
  \cny{USA}                                    % country
}
\address[id=aff3]{%                           % unique id
  \orgname{Department of Management Information Systems, Faculty of Economics and Administrative Sciences}, % university, etc
  \street{Tuzla},                     %
  %\postcode{}                                % post or zip code
  \city{Istanbul},                              % city
  \cny{Turkey}                                    % country
}
\address[id=aff4]{%                           % unique id
  \orgname{Faculty of Science and Letters, Piri Reis University}, % university, etc
  \street{Tuzla},                     %
  %\postcode{}                                % post or zip code
  \city{Istanbul},                              % city
  \cny{Turkey}                                    % country
}

%\address[id=aff2]{%
  %\orgname{Management and Information Systems Department, Piri Reis University},
  %\street{Tuzla},
  %\postcode{24105}
  %\city{Istanbul},
  %\cny{Turkey}
%}

%%%%%%%%%%%%%%%%%%%%%%%%%%%%%%%%%%%%%%%%%%%%%%
%%                                          %%
%% Enter short notes here                   %%
%%                                          %%
%% Short notes will be after addresses      %%
%% on first page.                           %%
%%                                          %%
%%%%%%%%%%%%%%%%%%%%%%%%%%%%%%%%%%%%%%%%%%%%%%

%\note{Sample of title note}     % note to the article
%\note[id=n1]{Equal contributor} % note, connected to author

\end{fmbox}% comment this for two column layout

%%%%%%%%%%%%%%%%%%%%%%%%%%%%%%%%%%%%%%%%%%%%%%
%%                                          %%
%% The Abstract begins here                 %%
%%                                          %%
%% Please refer to the Instructions for     %%
%% authors on http://www.biomedcentral.com  %%
%% and include the section headings         %%
%% accordingly for your article type.       %%
%%                                          %%
%%%%%%%%%%%%%%%%%%%%%%%%%%%%%%%%%%%%%%%%%%%%%%

\begin{abstractbox}

\begin{abstract} % abstract
In this study, we explore the properties of certain solutions of Hofstadter's famous $V$-recurrence, defined by the nested recurrence relation $V(n)=V(n-V(n-1))+V(n-V(n-4))$. First, we discover the nature behind a finite chaotic meta-Fibonacci sequence in terms of mortality in the $V$-recurrence. Then, we construct a new kind of quasi-periodic solution which suggests a connection with another Hofstadter-Huber recursion, $H(n)= H(n-H(n-2)) + H(n-H(n-3))$.
\end{abstract}

%%%%%%%%%%%%%%%%%%%%%%%%%%%%%%%%%%%%%%%%%%%%%%
%%                                          %%
%% The keywords begin here                  %%
%%                                          %%
%% Put each keyword in separate \kwd{}.     %%
%%                                          %%
%%%%%%%%%%%%%%%%%%%%%%%%%%%%%%%%%%%%%%%%%%%%%%


% MSC classifications codes, if any
%\begin{keyword}[class=AMS]
%\kwd[Primary ]{}
%\kwd{}
%\kwd[; secondary ]{}
%\end{keyword}

\end{abstractbox}
%
%\end{fmbox}% uncomment this for twcolumn layout

\end{frontmatter}

%%%%%%%%%%%%%%%%%%%%%%%%%%%%%%%%%%%%%%%%%%%%%%
%%                                          %%
%% The Main Body begins here                %%
%%                                          %%
%% Please refer to the instructions for     %%
%% authors on:                              %%
%% http://www.biomedcentral.com/info/authors%%
%% and include the section headings         %%
%% accordingly for your article type.       %%
%%                                          %%
%% See the Results and Discussion section   %%
%% for details on how to create sub-sections%%
%%                                          %%
%% use~\cite{...} to cite references        %%
%% ~\cite{koon} and                         %%
%% ~\cite{oreg,khar,zvai,xjon,schn,pond}    %%
%%  \nocite{smith,marg,hunn,advi,koha,mouse}%%
%%                                          %%
%%%%%%%%%%%%%%%%%%%%%%%%%%%%%%%%%%%%%%%%%%%%%%

%%%%%%%%%%%%%%%%%%%%%%%%% start of article main body
% <put your article body there>

%%%%%%%%%%%%%%%%
%% Background %%
%%
\section{Introduction}
In the 1960s, Douglas Richard Hofstadter introduced the $Q$-sequence~\cite{18}, which is defined by the nested recurrence relation $Q\p{n}=Q\p{n-Q\p{n-1}}+Q\p{n-Q\p{n-2}}$ with initial conditions $Q\p{1}=Q\p{2}=1$. The resulting sequence appears to grow approximately like $\frac{n}{2}$, but with a lot of noise. It remains open whether the $Q$-sequence actually grows this way, or, in fact, whether the sequence is truly infinite. It is theoretically possible that, at some point, $Q\p{n}$ could exceed $n$. If this happens, $Q\p{n+1}$ and all subsequent terms would be undefined. If a sequence generated by a nested recurrence is finite in this way, we say the sequence \emph{dies}.

Due to a superficial resemblance to the definition of the Fibonacci sequence, sequences defined like the $Q$-sequence are known as \emph{meta-Fibonacci sequences}. 
Hofstadter and Greg Huber have since~\cite{1} studied the two-parameter generalization of meta-Fibonacci recursions: $Q_{r,s}(n)=Q_{r,s}(n-Q_{r,s}(n-r)) + Q_{r,s}(n-Q_{r,s}(n-s))$ with $r<s$. Based on this investigation, a well-behaved solution to the recurrence $V(n)=V(n-V(n-1))+V(n-V(n-4))$ was discovered empirically. The initial conditions $V(1) = V(2) = V(3) = V(4) = 1$ generate a monotone solution that includes every positive integer, a property now known as \emph{slow}. Later, the properties of this solution were confirmed with a proof~\cite{3,2}. During the process of this investigation, a variety of experiments on the V-recurrence were carried out in order to understand the behaviour of other probable solution sequences~\cite{2}. However, very little is known about the behaviour of the $V$-recurrence under different sets of initial conditions. This study aims to clarify the properties of other solutions and their curious connections with another nested recursion: $H(n)= H(n-H(n-2)) + H(n-H(n-3))$.
 %\cite{koon,oreg,khar,zvai,xjon,schn,pond,smith,marg,hunn,advi,koha,mouse}
 
 \subsection{Notation}
Going forward, un-decorated symbols such as $Q$, $V$, etc.\ will be used to denote specific sequences with those names. To refer to other sequences satisfying the same recurrences, we use the symbols with subscripts. Also, initial conditions are often denoted in this paper by sequences of numbers enclosed in angle brackets. For example, the initial conditions to the $V$-sequence would be written as $\ic{1,1,1,1}$.

\section{On Mortality of Remarkably Long Life}
Finding nested recurrence relations with increasing “mortality” and understanding the generational behaviour of them can be seen as a meaningful attempt in order to discover the nature of chaotic solutions~\cite{4}. In literature, there are some examples of long-living, finite, chaotic sequences which are produced by meta-Fibonacci recurrences. One remarkable example for the $V$-recurrence is that the initial conditions $\ic{3,1,4,4}$ generates a sequence that terminates after $474767$ terms~\cite{2}. Similarly, the recurrence $B_A(n)=B_A(n-B_A(n-1))+B_A(n-B_A(n-2))+B_A(n-B_A(n-3))$ with the initial conditions $\ic{1,1,1,4,3}$ dies~\cite{5} when $B_A(509871) = 519293$. More surprisingly, Isgur notes that $L_A(n)=L_A(n-19-L_A(n-3))+L_A(n-28-L_A(n-12))$ with initial conditions $L_A(n)=1$ for $1 \leq n \leq 29$ has relatively long life and it becomes incalculable after more than $19$ million terms~\cite{6}. Inspired by these curious examples, we study an exceptional chaotic sequence~\cite{7} $V_c$ that is generated by the $V$-recurrence with the initial conditions $\ic{3,4,5,4,5,6}$. Investigation of the behaviour of $V_{c}(n)$ may be highly illustrative since it has really long life~\cite{7} (more precisely $V_{c}(3080193026) = V_{c}(3080193026 - V_{c}(3080193025)) +  V_{c}(3080193026 - V_{c}(3080193022)) =V_{c}(2290654567) + V_{c}(1873687422) = 1686223049 + 1415176819 = 3101399868$), and it has a curious generational structure.

\subsection{Generational Structure}
Before we proceed, it is important to discuss the concept of a generational structure. A sequence generated by a two-term Hofstadter-Huber recurrence (such as the $Q$ or $V$-recurrence) has the property that each term is the sum of two earlier terms in the sequence. The indices of these earlier terms are known as the \emph{parents} of the current term, with the index coming from the first term in the recurrence known as the \emph{mother} or \emph{mother spot} and the second as the \emph{father} or \emph{father spot}. For example, in a sequence $V_a$ satisfying the $V$-recurrence, the mother of $V_a\p{n}$ is $n-V_a\p{n-1}$ and the father is $n-V_a\p{n-4}$.

Some sequences have the property that they can be partitioned into intervals where both parents of a term in one interval lie in the previous interval. If this is possible, the sequence is said to have a \emph{generational structure}. Sometimes, this is not possible, but it is almost possible. It is also useful to discuss generations in these cases~\cite{11,8,19,9,10}.

\subsection{Generations of $V_c(n)$}

In order to see the facts behind this long-lived finite sequence, we need to construct some auxilary sequences which analyze the generational structure of $V_{c}(n)$. With similar methodology of previous studies~\cite{11,8, 9}, we can define $W(n)$, $P_{s}(n)$ and $R(n)$ as below, see Table 1 and Table 2 for corresponding values. In our experimental range, these auxilary sequences are used in order to detect unpredictable sub-generations of the sequence which are responsible for termination of $V_{c}(n)$. See Figure 1 in order to observe generational boundaries of $V_{c}(n)$.

\newpage

\begin{definition}

Let $W(n)$ be the least $m$ such that minimum of the father $(m-V_{c}(m-4))$ and mother $(m-V_{c}(m-1))$ spots is equal or greater than $n$.

\end{definition}

\begin{definition}

Let $P_{s}(n) = W(P_{s}(n-1))$ for $n > 2$, with $P_{s}(1) = 1$ and $P_{s}(2) = 4$. Furthermore, define $P(n) = P_{s}(n) + 3$ for $n > 2$, with $P(1) = P_{s}(1)$ and $P(2) = P_{s}(2)$.

\end{definition}

\begin{definition}

Let $R(n)$ be the largest $m < P(n+1)-1$ such that $V_{c}(m+1) - V_{c}(m)$ is not 0 or 1 for $n > 2$, with $R(1) = 1$ and $R(2) = 4$.

\end{definition}

For a corresponding noise sequence, define $S_c(n) = V_c(n)-\frac{n}{2}$, Let $\big \langle S_{c}(n) \big \rangle_{k}$ denote the average value of  $S_{c}(n)$ over the $k^{th}$ generation's boundaries that are determined by $P(k)$ and $R(k)$, and define $\alpha(k, S_{c}(n))$ as below. 
\\

%\begin{itemize}
%\item
%$M_{k}(S_{c}(n))^2 = \big \langle S_{c}(n)^2 \big \rangle_{k} - \big \langle S_{c}(n) \big \rangle_{k}^2$
%\\
%\item
%$\alpha(k, S_{c}(n)) = \log_2\!\left(\frac{M_{k}(S_{c}(n))}{M_{k-1}(S_{c}(n))}\right)$\label{p:alpha}
%\end{itemize}

\begin{equation}
\begin{cases}
M_{k}(S_{c}(n))^2 = \big \langle S_{c}(n)^2 \big \rangle_{k} - \big \langle S_{c}(n) \big \rangle_{k}^2\\
\\
\alpha(k, S_{c}(n)) = \log_2\!\left(\frac{M_{k}(S_{c}(n))}{M_{k-1}(S_{c}(n))}\right)\label{p:alpha}
\end{cases}
\end{equation}
%\newpage

\begin{figure}[!hb]
\begin{center}
\includegraphics[width=0.7\textwidth]{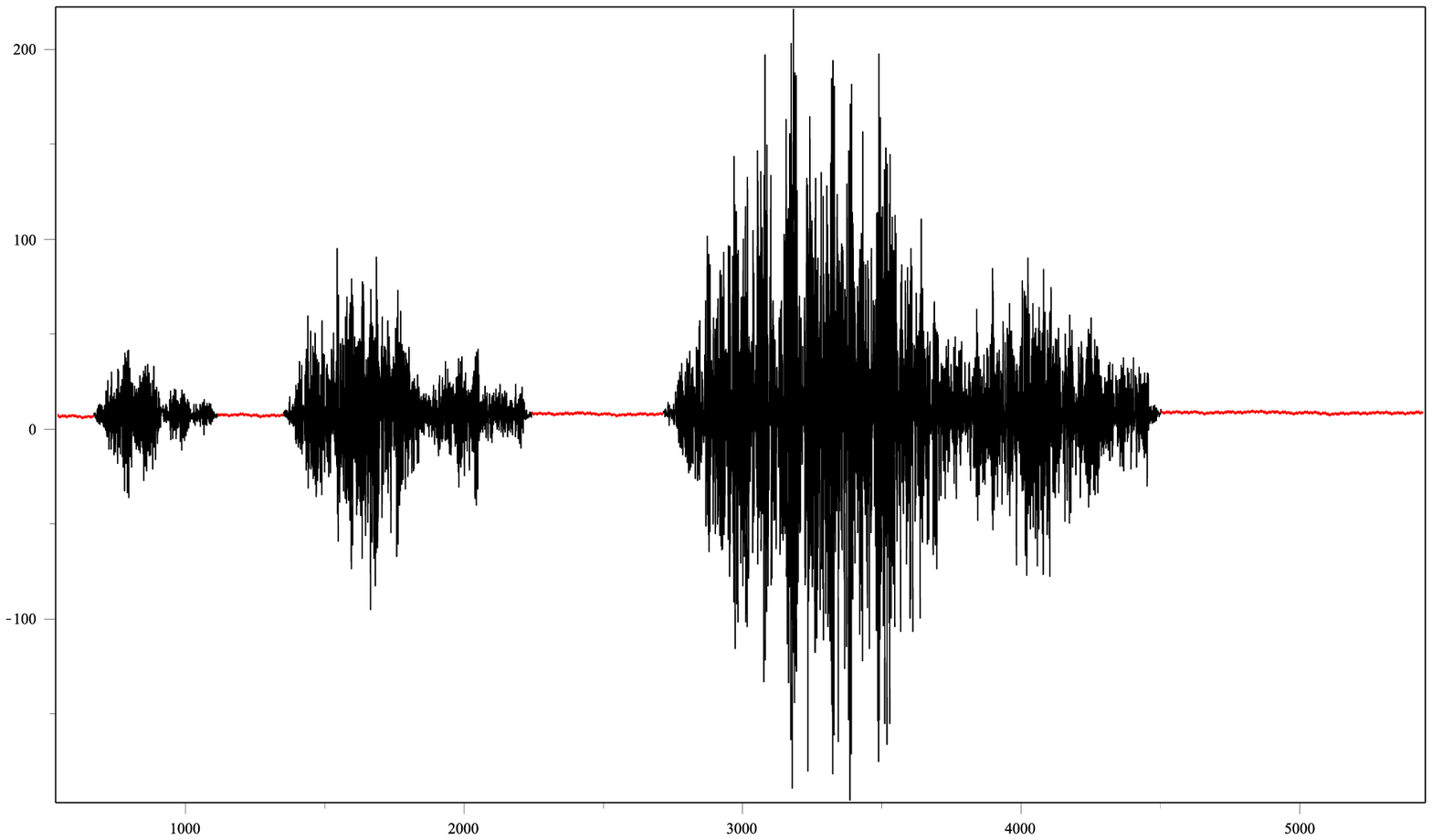}
\caption{A line plot of $S_c(n) = V_c(n) - \frac{n}{2}$ for $R(7) \leq n \leq P(11)$. Red regions are corresponding to slow subsequence of $V_c(n)$, while black regions have unpredictable noise characteristics.}
\label{fig:1122}   % Give a unique label
\end{center}
\end{figure}

\begin{table*}[!]
\caption {The values of $P(n)$ sequence for $n \leq 20$.}
\begin{center}

%\begin{ruledtabular}
\begin{adjustbox}{max width=1\textwidth}
\begin{tabular}{cccccc}

\noalign{\smallskip}
$ $ & $ $ & $ $ & $m$ & $ $ & $ $\\
\noalign{\smallskip}\hline\noalign{\smallskip}
\noalign{\smallskip}
$ $ & $1$ & $2$ & $3$ & $4$ & $5$\\
\noalign{\smallskip}\hline\noalign{\smallskip}
$P(m+0)$ &1	&4	 &17	 &37 &78\\
$P(m+5)$ &162	 &331 &671 &1352&2715\\
$P(m+10)$ &5443 &10900 &21816	&43649	&87316\\
$P(m+15)$ &174652&349325	&698673	&1397370	&2794765\\
\hline
\noalign{\smallskip}
\\
\\
\end{tabular}
\end{adjustbox}
\caption {The values of $R(n)$ sequence for $n \leq 20$.}
\label{tab:table1}
%\end{ruledtabular}
%\end{center}
%\end{table*}
%\begin{table*}[!htbp]
%\begin{center}
%\begin{ruledtabular}

%\end{center}
%\end{table*}

%\begin{table*}[!htb]
%\begin{center}
%\begin{ruledtabular}
\begin{adjustbox}{max width=1\textwidth}
\begin{tabular}{cccccc}
\noalign{\smallskip}
$ $ & $ $ & $ $ & $m$ & $ $ & $ $\\
\noalign{\smallskip}\hline\noalign{\smallskip}
\noalign{\smallskip}
$ $ & $1$ & $2$ & $3$ & $4$ & $5$\\
\noalign{\smallskip}\hline\noalign{\smallskip}
$R(m+0)$ &1	&4	 &18	 &45 &111\\
$R(m+5)$ &257	 &542 &1115 &2242&4501\\
$R(m+10)$ &9029 &18088 &36213	&72462	&144994\\
$R(m+15)$ &290027&580112	&1161200	&2323822	&4650379\\
\hline
\noalign{\smallskip}
\end{tabular}
\end{adjustbox}
\label{tab:tablex}
%\end{ruledtabular}
\end{center}
\end{table*}
\begin{table*}[!h]
\caption {Values of $\alpha(k, S_{c}(n))$ for $5 \leq k \leq 20$.}
\begin{center}
%\begin{ruledtabular}
\begin{adjustbox}{max width=0.85\textwidth}
\begin{tabular}{cc}
\noalign{\smallskip}
$k$ & $\alpha(k, S_{c}(n))$  \\
\noalign{\smallskip}\hline\noalign{\smallskip}

5	&0.8402	  \\
6	&0.7278	  \\
7	&0.7477	  \\
8	&1.4374	   \\
9	&1.0590	  \\
10	&1.1340	  \\
11	&1.1686	  \\
12	&1.1744  \\
13	&1.1077  \\
14	&1.1656	  \\
15	&1.1558	  \\
16   &1.1339 \\
17   &1.1371 \\
18   &1.1336 \\
19   &1.1212 \\
20   &1.1231 \\
\hline
\noalign{\smallskip}
\end{tabular}
\end{adjustbox}

\label{tab:stat3}
%\end{ruledtabular}
\end{center}
\end{table*}

Computational results in Table 3 show that $\alpha$ values oscillate in different range than the sequences which include Hofstadter's $Q$-sequence that  are studied before~\cite{11,8,9,10}. This investigation suggests that $V_{c}(n)$ is going to termination step by step in its successive generational order due to increasing characteristics of noise that $\alpha$ values which are greater than $1$ depict. So experimental evidence suggests that infinite chaotic solution for $V$-recurrence is very difficult to construct although exceptionally long life is possible for some choices of initial conditions sets based on appearance of such generational structure. Furthermore, initial condition patterns which are formulated by asymptotic property of $V$-sequence also confirm mortality of solution sequence if a slow solution does not exist ~\cite{11}. At this case, it is natural to think that if a non-slow infinite solution sequence exists for $V$-recurrence, most likely that solution has quasi-periodic nature ~\cite{13,15}.

\newpage

\section{A new kind of solution}
Given a meta-Fibonacci recurrence, there is a known algorithm to search for solutions to it that satisfy a linear recurrence relation~\cite{14}. This algorithm finds infinite families of solutions that eventually consist of interleavings of simple (typically constant or linear) subsequences. For the $V$-recurrence, the algorithm finds $20$ solution families eventually consisting of interleavings of five constant or linear sequences. Since solutions to the $V$-recurrence are invariant under shifting all of the terms~\cite{12}, this corresponds to four fundamentally different families. The initial conditions in Table~\ref{tab:icv} each generate a representative of one of these families. (Despite having the same constant-linear pattern, the terms in the last two families have different congruences mod $5$. They are therefore distinct families.)

\begin{table}
\begin{tabular}{|c|c|}\hline
\textbf{Pattern} & \textbf{Initial Condition}\\\hline
C,C,C,L,L & \[ \ic{5, 4, 0, 0, 0, 5, 0, 5, 5, 1, 5, 4}  \]\\\hline
C,C,L,C,L & \[ \ic{4, 0, 5, -2, 1, 3, -3, 5, 3, 0, 4, 10, 5, 8}  \]\\\hline
C,L,C,L,L & \[ \ic{0, 14, -4, -7, 8, 5, 14, -2, -2, 8, 0, 0, 6, 3, 18, 15, 14, 11, 8, 8, 20, 14, 16, 13, 8, 25}  \]\\\hline
C,L,C,L,L & \[ \ic{0, 2, -2, -6, 11, 6, 2, 3, 0, 11, 0, 2, 8, -2, 11, 15, 2, 13}  \]\\\hline
\end{tabular}
\caption{Patterns and representative initial conditions for each of the four families of period-$5$ solutions to the $V$-recurrence. (C=constant, L=linear)}
\label{tab:icv}
\end{table}

In this section, we describe another infinite family of solutions to the $V$-recurrence. Like the families in Table~\ref{tab:icv}, its members eventually consist of five relatively simple interleaved sequences. But, as we shall see, not all of them are constant or linear. Then, we describe a related family of solutions to a different recurrence.

\subsection{A System of Nested Recurrences with Slow Solutions}

As an aside, we first discuss the behavior of solutions to a certain type of system of nested recurrences.
\begin{definition}
For integers $c_f$, $d_f$, $c_g$, and $d_g$ with $d_f+d_g>0$, the \emph{Golomb-like system} with those parameters is the system
\[
\begin{cases}
f\p{n}=g\p{n-g\p{n-1}-c_f}+d_f\\
g\p{n}=f\p{n-f\p{n}-c_g}+d_g.
\end{cases}
\]
\end{definition}
\noindent
The name \emph{Golomb-like} stems from the observation that the recurrences in these systems bear a superficial resemblance to Golomb's~\cite{16} recurrence $G\p{n}=G\p{n-G\p{n-1}}+1$. Also, solutions to Golomb-like systems appear to behave similarly to solutions to Golomb's recurrence. In particular, Golomb-like systems have some slow solutions with simple descriptions, which we will see shortly. They also appear to have many non-slow solutions with noticeable patterns. Golomb's recurrence exhibits similar behavior, and it is conjectured~\cite{17} that all solutions to Golomb's recurrence grow asymptotically like $\sqrt{2n}$. We have a similar conjecture about Golomb-like systems:
\begin{con}
Any solution to the Golomb-like system with parameters $c_f$, $d_f$, $c_g$, and $d_g$ grows asymptotically like $\sqrt{\pb{d_f+d_g}n}$.
\end{con}
\noindent
The evidence for this conjecture comes from experimentation combined with the similar behavior to Golomb's recurrence. In particular, these solutions all appear to be sub-linear.

%We typically begin our sequences derived from nested recurrences with index $1$, and will continue to do so in the examples immediately below. But, in the applications to the $V$-recurrence, it will be simpler if we allow solutions to Golomb-like systems to start from any fixed index $n_0$. Typically, $n_0$ will be a negative number. This is just a matter of convenience, as solutions to Golomb-like systems are invariant under shifting all of the terms~\cite{12}.

\subsubsection{Specific Solutions to Golomb-like Systems}
We now examine a few specific solutions to some Golomb-like systems. All of these solutions are slow and easy to describe. They all appear again in connections with the $V$-recurrence.
\begin{pro}\label{prop:fg1}
The Golomb-like system
\[
\begin{cases}
f\p{n}=g\p{n-g\p{n-1}}\\
g\p{n}=f\p{n-f\p{n}}+1
\end{cases}
\]
given initial condition $f\p{1}=0$ generates a slow solution where each nonnegative integer $i$ appears in the $f$-sequence $2i+1$ times and in the $g$-sequence $2i$ times.
\end{pro}
\begin{proof}
If each nonnegative integer $i$ appears $2i+1$ times in the $f$-sequence, terms $f\p{i^2+1}$ through $f\p{i^2+2i+1}$ must equal $i$. Similarly, if each nonnegative integer $i$ appears $2i$ times in the $g$-sequence, terms $g\p{i^2-i+1}$ through $g\p{i^2+i}$ must equal $i$.

We now proceed by induction on the index. First, we observe that $f\p{1}=0$, as required. 
%, $g\p{0}=f\p{0-f\p{0}}+1=1$, and $f\p{1}=g\p{1-g\p{0}}=1$. So, the $f$-sequence has one $0$, and the $g$-sequence has no $0$'s, as required. 
%We can compute additional terms of both sequences to observe that the $g$-sequence begins with exactly two $1$'s.
%
%Now, suppose $i\geq1$, and suppose that each nonnegative integer less than $i$ appears the appropriate number of times in the $f$-sequence, and suppose each nonnegative integer less than $i+2$ appears the appropriate number of times in the $g$-sequence.  We now examine occurrences of $i$ in the $f$-sequence and $i+2$ in the $g$-sequence.
%
We now examine each sequence, starting with the $g$-sequence. Suppose $n$ is a positive integer, and suppose that, for all $m\leq n$, $f\p{m}$ equals its desired value. 
%First, we wish to show that all terms of the form $f\p{i^2+1+r}$ equal $i$, where $0\leq r<2i+1$. 
We can write $n=i^2-i+1+r$ for some $i\geq1$ and $0\leq r<2i$. Wishing to show $g\p{n}=i$, we have
\begin{align*}
g\p{n}&=g\p{i^2-i+1+r}\\
&=f\p{i^2-i+1+r-f\p{i^2-i+1+r}}+1\\
&=f\p{i^2-i+1+r-f\p{i^2+1+\pb{r-i}}}+1.
\end{align*}
We now have two cases to consider:
\begin{description}
\item[$r<i$:] If $r<i$, then $r-i<0$, meaning that $f\p{i^2+1+\pb{r-i}}=i-1$ by induction. We then have that
\begin{align*}
g\p{n}&=f\p{i^2-i+1+r-\pb{i-1}}\\
&=f\p{i^2-2i+2+r}+1\\
&=f\p{\pb{i-1}^2+1+r}+1.
\end{align*}
Since $0\leq r<i<2\pb{i-1}+1$, we have that $f\p{\pb{i-1}^2+1+r}=i-1$, meaning $g\p{n}=i$, as required.
%%%%%%
\item[$r\geq i$:] If $r\geq i$, then $r-i\geq0$, meaning that $f\p{i^2+1+\pb{r-i}}=i$ by induction. We then have that
\begin{align*}
g\p{n}&=f\p{i^2-i+1+r-i}\\
&=f\p{i^2-2i+1+r}+1\\
&=f\p{\pb{i-1}^2+1+\pb{r-1}}+1.
\end{align*}
Since $i-1\leq r-1<i<2\pb{i-1}+1$, we have that $f\p{\pb{i-1}^2+1+\pb{r-1}}=i-1$, meaning $g\p{n}=i$, as required.
\end{description}

Now, we examine the $f$-sequence. Suppose $n\geq2$ is an integer, and suppose that, for all $m<n$, $g\p{m}$ equals its desired value. 
%First, we wish to show that all terms of the form $f\p{i^2+1+r}$ equal $i$, where $0\leq r<2i+1$. 
We can write $n=i^2+1+r$ for some $i\geq1$ and $0\leq r<2i+1$. Wishing to show $f\p{n}=i$, we have
\begin{align*}
f\p{n}&=f\p{i^2+1+r}\\
&=g\p{i^2+1+r-g\p{i^2+1+r}}\\
&=g\p{i^2+1+r-g\p{i^2+i+1+\pb{r-i-1}}}\\
&=g\p{i^2+1+r-g\p{\pb{i+1}^2-\pb{i+1}+1+\pb{r-i-1}}}.
\end{align*}
We now have two cases to consider:
\begin{description}
\item[$r\leq i$:] If $r\leq i$, then $r-i-1<0$, meaning that $g\p{\pb{i+1}^2-\pb{i+1}+1+\pb{r-i-1}}=i$ by induction. We then have that
\begin{align*}
f\p{n}&=g\p{i^2+1+r-i}\\
&=g\p{i^2-i+1+r}.
\end{align*}
Since $0\leq r\leq i<2i$, we have that $g\p{i^2-i+1+r}=i$, meaning $f\p{n}=i$, as required.
%%%%%%
\item[$r>i$:] If $r> i$, then $r-i-1\geq0$, meaning that $g\p{\pb{i+1}^2-\pb{i+1}+1+\pb{r-i-1}}=i+1$ by induction. We then have that
\begin{align*}
f\p{n}&=g\p{i^2+1+r-\pb{i+1}}\\
&=g\p{i^2-i+1+\pb{r-1}}.
\end{align*}
Since $i\leq r-1<2i$, we have that $g\p{i^2-i+1+\pb{r-1}}=i$, meaning $f\p{n}=i$, as required.
\end{description}
\end{proof}

\begin{pro}\label{prop:fg2}
The Golomb-like system
\[
\begin{cases}
f\p{n}=g\p{n-g\p{n-1}}\\
g\p{n}=f\p{n-f\p{n}}+2
\end{cases}
\]
given initial conditions $f\p{1}=0$, $f\p{2}=1$, $f\p{3}=1$, $g\p{1}=1$, $g\p{2}=1$, and $g\p{3}=2$ generates a slow solution where:
%each odd integer $i$ appears in the $f$-sequence $2i+1$ times and in the $g$-sequence $2i$ times.
\begin{itemize}
\item Each odd integer $i\geq3$ appears in the $f$-sequence $2i+1$ times and the $g$-sequence $2i-1$ times.
\item Each even positive integer appears in each sequence once.
\item The $f$-sequence starts with $0$, this being the only appearance of $0$ in either sequence.
\item The number $1$ appears exactly $4$ times in the $f$-sequence and exactly twice in the $g$-sequence.
\end{itemize}
\end{pro}
Proposition~\ref{prop:fg2} has a similar proof to Proposition~\ref{prop:fg1}, with the added complication of keeping track of even versus odd. The odd terms are generated similarly to all the terms in the proof of Proposition~\ref{prop:fg1}, and each even term in one sequence comes from the preceding even term in the other sequence. For brevity, the proof of Proposition~\ref{prop:fg2} is omitted.

\begin{pro}\label{prop:fg12}
The Golomb-like system
\[
\begin{cases}
f\p{n}=g\p{n-g\p{n-1}}+1\\
g\p{n}=f\p{n-f\p{n}}+2
\end{cases}
\]
given initial conditions $f\p{1}=0$, $f\p{2}=1$, $f\p{3}=1$, $g\p{1}=1$, $g\p{2}=1$, and $g\p{3}=2$ generates a slow solution where:
%each odd integer $i$ appears in the $f$-sequence $2i+1$ times and in the $g$-sequence $2i$ times.
\begin{itemize}
\item If $i\geq3$ is a multiple of $3$, $i$ appears once in the $f$-sequence and $2i-2$ times in the $g$-sequence.
\item If $i\geq4$ is congruent to $1$ mod $3$, $i$ appears $2i-1$ times in the $f$-sequence and twice in the $g$-sequence.
\item If $i$ is a positive integer congruent to $2$ mod $3$, $i$ appears twice in the $f$-sequence and once in the $g$-sequence.
\item The $f$-sequence starts with $0$, this being the only appearance of $0$ in either sequence.
\item The number $1$ appears exactly twice in each sequence.
\end{itemize}
\end{pro}
Again, the proof is similar and is omitted for brevity.

\subsection{An Infinite Family of Solutions to the $V$-recurrence}

We are now able to describe an infinite family of solutions to the $V$-recurrence that consist of interleavings of five simpler sequences. These solutions are of a similar flavor to those in~\cite{14}. But, the methods of that paper would not find these solutions, as these include subsequences that are $\Theta\p{\sqrt{n}}$ in growth.
\begin{thm}\label{thm:infv}
Let $K$, $b_0$, $b_1$, $b_2$, $b_4$, $a_f$, $a_g$, and $m$ be integers satisfying the following properties:
%\begin{multicols}{2}
\begin{itemize}
\item $b_0\equiv1\pmod{5}$
\item $b_1\equiv4\pmod{5}$
\item $b_2\equiv2\pmod{5}$ and $7\leq b_2<K+3$
\item $b_4\equiv3\pmod{5}$ and $8\leq b_4<K+5$
\item $a_f\equiv 2\pmod{5}$
\item $a_g\equiv 3\pmod{5}$
\item $a_f+a_g>0$
\item $m\geq1$.
\end{itemize}
%\end{multicols}
\noindent
Define the following Golomb-like system:
\[
\begin{cases}
f\p{n}=g\!\pb{n-g\p{n-1}-\frac{b_1+1}{5}}+\frac{b_1-b_0+a_f}{5}\\
g\p{n}=f\!\pb{n-f\p{n}-\frac{b_0-1}{5}}+\frac{b_0-b_1+a_g}{5}.
\end{cases}
\]
%where the first terms in the sequences defined by these recurrences are at index $n_0=-\fl{\frac{K-1}{5}}$. 
Then, there is a solution $V_G$ to the $V$-recurrence that, starting at index $K$, has the form
\[
\begin{cases}
V_G\p{K+5k}=5f\p{k}+b_0\\
V_G\p{K+5k+1}=5g\p{k}+b_1\\
V_G\p{K+5k+2}=5k+b_2\\
V_G\p{K+5k+3}=5m\\
V_G\p{K+5k+4}=5k+b_4
\end{cases}
\]
with any initial condition satisfying the following properties:
\begin{enumerate}
\item\label{it:af} $V_G\p{K+5-b_4}=a_f$
\item\label{it:ag} $V_G\p{K+6-b_2}=a_g$
\item\label{it:5m} $V_G\p{K+3-b_2}+V_G\p{K+8-b_4}=5m$
\item\label{it:b2} For each integer $1\leq i\leq m$, $V_G\p{K+2-5i}=b_2-5i$
%\item\label{it:b3} For each integer $1\leq i\leq m$, $V_G\p{K+3-5i}=5m$
\item\label{it:b3} $V_G\p{K-2}=5m$
\item\label{it:b4} For each integer $1\leq i\leq m$, $V_G\p{K+4-5i}=b_4-5i$
%\item Let $p_1,p_2,\ldots,p_t$ denote the sequence of terms in the initial condition at indices congruent to $K$ mod $5$, and let $q_1,q_2,\ldots,q_t$ denote the sequence of terms in the initial condition at indices congruent to $K+1$ mod $5$ (excluding $V_G\p{1}$ if relevant). The sequences $\frac{p_1-b_0}{5},\frac{p_2-b_0}{5},\ldots,\frac{p_t-b_0}{5}$ and $\frac{q_1-b_1}{5},\frac{q_2-b_1}{5},\ldots,\frac{q_t-b_1}{5}$, when fed to the system for $f$ and $g$
\item\label{it:ic} Let $n_0=\fl{\frac{K-1}{5}}$. The initial conditions
\[
\ic{\frac{V_G\p{K-5n_0}-b_0}{5}, \frac{V_G\p{K-5\pb{n_0-1}}-b_0}{5}, \frac{V_G\p{K-5\pb{n_0-2}}-b_0}{5}, \ldots, \frac{V_G\p{K-5}-b_0}{5}}
\]
for the $f$-sequence and
\[
\ic{\frac{V_G\p{K-5n_0+1}-b_1}{5}, \frac{V_G\p{K-5\pb{n_0-1}+1}-b_1}{5}, \frac{V_G\p{K-5\pb{n_0-2}+1}-b_1}{5}, \ldots, \frac{V_G\p{K-4}-b_1}{5}}
\]
for the $g$-sequence generate sublinear sequences where, for all $n>n_0$, $f\p{n}\leq n+\frac{b_0-1}{5}$, $g\p{n-1}\leq n+\frac{b_1+1}{5}$, and $g\p{n}\leq n+\frac{b_1+1}{5}$. (Note that not all terms in these initial conditions need to be integers, but, in order for the sequences to live, no term after the initial condition can refer to a non-integer term.)
\end{enumerate}
\end{thm}

\begin{proof}
The proof is by induction on the index, with base case provided by the initial condition. Suppose the parameters and initial conditions satisfy all of the listed conditions, and furthermore suppose that the general form of the solution holds through index $n-1$ for some $n\geq K$. We now have five cases to consider:
\begin{description}
\item[$n-K\equiv 0\pmod{5}$:]
In this case, $n=K+5k$ for some $k\geq0$. We have
\[
V_G\p{K+5k}=V_G\p{K+5k-V_G\p{K+5k-1}}+V_G\p{K+5k-V_G\p{K+5k-4}}.
\]
By induction, $V_G\p{K+5k-1}=5\pb{k-1}+b_4$, as this term either falls in the inductive hypothesis or in restriction~\ref{it:b4} on the initial condition. Similarly, $V_G\p{K+5k-4}=5g\p{k-1}+b_1$, as this term either falls in the inductive hypothesis or in restriction~\ref{it:ic} on the initial condition. So, we have
\begin{align*}
V_G\p{K+5k}&=V_G\p{K+5k-5\pb{k-1}-b_4}+V_G\p{K+5k-5g\p{k-1}-b_1}\\
&=V_G\p{K+5-b_4}+V_G\p{K+5\pb{k-g\p{k-1}}-b_1}.
\end{align*}
We now observe that $V_G\p{K+5-b_4}=a_f$ by restriction~\ref{it:af} on the initial condition. Also, since $b_1\equiv4\pmod{5}$ and since restriction~\ref{it:ic} guarantees $g\p{k-1}\leq k+\frac{b_1+1}{5}$, we have $V_G\p{K+5\pb{k-g\p{k-1}}-b_1}=5g\!\pb{k-g\p{k-1}-\frac{b_1+1}{5}}+b_1$. Putting these together yields
\begin{align*}
V_G\p{K+5k}&=a_f+5g\!\pb{k-g\p{k-1}-\frac{b_1+1}{5}}+b_1\\
&=5g\!\pb{k-g\p{k-1}-\frac{b_1+1}{5}}+\pb{b_1-b_0+a_f}+b_0\\
&=5f\p{k}+b_0,
\end{align*}
as required.
%%%%%%%%%%%%%%%
\item[$n-K\equiv 1\pmod{5}$:]
In this case, $n=K+5k+1$ for some $k\geq0$. We have
\[
V_G\p{K+5k+1}=V_G\p{K+5k+1-V_G\p{K+5k}}+V_G\p{K+5k+1-V_G\p{K+5k-3}}.
\]
By induction, $V_G\p{K+5k-3}=5\pb{k-1}+b_2$, as this term either falls in the inductive hypothesis or in restriction~\ref{it:b2} on the initial condition. Similarly, $V_G\p{K+5k}=5f\p{k}+b_0$, as this term falls in the inductive hypothesis. So, we have
\begin{align*}
V_G\p{K+5k+1}&=V_G\p{K+5k+1-5f\p{k}-b_0}+V_G\p{K+5k+1-5\pb{k-1}-b_2}\\
&=V_G\p{K+5\pb{k-f\p{k}}+1-b_0}+V_G\p{K+6-b_2}.
\end{align*}
We now observe that $V_G\p{K+6-b_2}=a_g$ by restriction~\ref{it:ag} on the initial condition. Also, since $b_0\equiv1\pmod{5}$ and since restriction~\ref{it:ic} guarantees $f\p{k}\leq k+\frac{b_0-1}{5}$, we have $V_G\p{K+5\pb{k-f\p{k}}-b_0}=5f\!\pb{k-f\p{k}-\frac{b_0-1}{5}}+b_0$. Putting these together yields
\begin{align*}
V_G\p{K+5k+1}&=5f\!\pb{k-f\p{k}-\frac{b_0-1}{5}}+b_0+a_g\\
&=5f\!\pb{k-f\p{k}-\frac{b_0-1}{5}}+\pb{b_0-b_1+a_g}+b_1\\
&=5g\p{k}+b_1,
\end{align*}
as required.
%%%%%%%%%%%%%%%
\item[$n-K\equiv 2\pmod{5}$:]
In this case, $n=K+5k+2$ for some $k\geq0$. We have
\[
V_G\p{K+5k+2}=V_G\p{K+5k+2-V_G\p{K+5k+1}}+V_G\p{K+5k+2-V_G\p{K+5k-2}}.
\]
By induction, $V_G\p{K+5k+1}=5g\p{k}+b_1$, as this term falls in the inductive hypothesis. Similarly, $V_G\p{K+5k-2}=5m$, as this term either falls in the inductive hypothesis or in restriction~\ref{it:b3} on the initial condition. So, we have
\begin{align*}
V_G\p{K+5k+2}&=V_G\p{K+5k+2-5g\p{k}-b_1}+V_G\p{K+5k+2-5m}\\
&=V_G\p{K+5\pb{k-g\p{k}}+2-b_1}+V_G\p{K+5k+2-5m}.
\end{align*}
We now observe that $V_G\p{K+5k+2-5m}=5\pb{k-m}+b_2$, as this term falls in the inductive hypothesis or in restriction~\ref{it:b2} on the initial condition. Also, since $b_1\equiv4\pmod{5}$ and since restriction~\ref{it:ic} guarantees $g\p{k}\leq k+\frac{b_1+1}{5}$, we have $V_G\p{K+5\pb{k-g\p{k}}+2-b_1}=5m$. Putting these together yields
\begin{align*}
V_G\p{K+5k+2}&=5m+5\pb{k-m}+b_2=5k+b_2,
\end{align*}
as required.
%%%%%%%%%%%%%%%
\item[$n-K\equiv 3\pmod{5}$:]
In this case, $n=K+5k+3$ for some $k\geq0$. We have
\[
V_G\p{K+5k+3}=V_G\p{K+5k+3-V_G\p{K+5k+2}}+V_G\p{K+5k+3-V_G\p{K+5k-1}}.
\]
By induction, $V_G\p{K+5k+2}=5k+b_2$, as this term falls in the inductive hypothesis. Similarly, $V_G\p{K+5k-1}=5\pb{k-1}+b_4$, as this term either falls in the inductive hypothesis or in restriction~\ref{it:b4} on the initial condition. So, we have
\begin{align*}
V_G\p{K+5k+3}&=V_G\p{K+5k+3-5k-b_2}+V_G\p{K+5k+3-5\pb{k-1}-b_4}\\
&=V_G\p{K+3-b_2}+V_G\p{K+8-b_4}.
\end{align*}
By restriction~\ref{it:5m}, this equals $5m$, as required.
%%%%%%%%%%%%%%%
\item[$n-K\equiv 4\pmod{5}$:]
In this case, $n=K+5k+4$ for some $k\geq0$. We have
\[
V_G\p{K+5k+4}=V_G\p{K+5k+4-V_G\p{K+5k+3}}+V_G\p{K+5k+4-V_G\p{K+5k}}.
\]
By induction, $V_G\p{K+5k}=5f\p{k}+b_0$, as this term falls in the inductive hypothesis. Similarly, $V_G\p{K+5k+3}=5m$, as this term also falls in the inductive hypothesis. So, we have
\begin{align*}
V_G\p{K+5k+4}&=V_G\p{K+5k+4-5m}+V_G\p{K+5k+4-5f\p{k}-b_0}\\
&=V_G\p{K+5k+4-5m}+V_G\p{K+5\pb{k-f\p{k}}+4-b_0}.
\end{align*}
We now observe that $V_G\p{K+5k+4-5m}=5\pb{k-m}+b_4$, as this term falls in the inductive hypothesis or in restriction~\ref{it:b4} on the initial condition. Also, since $b_0\equiv1\pmod{5}$ and since restriction~\ref{it:ic} guarantees $f\p{k}\leq k+\frac{b_0-1}{5}$, we have $V_G\p{K+5\pb{k-f\p{k}}+4-b_0}=5m$. Putting these together yields
\begin{align*}
V_G\p{K+5k+4}&=5\pb{k-m}+b_4+5m=5k+b_4,
\end{align*}
as required.
\end{description}
\end{proof}
\subsubsection{Concrete Examples of Solutions to the $V$-recurence}
Let us now see a couple of concrete solutions to the $V$-recurrence corresponding to specific settings of the parameters in Theorem~\ref{thm:infv}.
\begin{pro}\label{prop:v1}
The initial conditions $\ic{4,2,5,3,1}$ to the Hofstadter $V$-recurrence produce a solution of the following form for $k\geq1$:
\[
\begin{cases}
V_G\p{5k}=5f\p{k}+1\\
V_G\p{5k+1}=5g\p{k}-1\\
V_G\p{5k+2}=5k+2\\
V_G\p{5k+3}=5\\
V_G\p{5k+4}=5k+3,
\end{cases}
\]
where $f$ and $g$ are the sequences in Proposition~\ref{prop:fg1}.
\end{pro}
\begin{proof}
Let $K=10$, $b_0=1$, $b_1=-1$, $b_2=12$, $b_4=13$, $a_f=2$, $a_g=3$, and $m=1$. These values satisfy all of the requirements on these parameters. The first nine terms of the sequence resulting from the initial conditions $\ic{4,2,5,3,1}$ are $4,2,5,3,1,4,7,5,8$. We now check that these satisfy all seven requirements:
\begin{itemize}
\item We have $V_G\p{K+5-b_4}=V_G\p{10+5-13}=V_G\p{2}=2=a_f$, as required.
\item We have $V_G\p{K+6-b_2}=V_G\p{10+6-12}=V_G\p{4}=3=a_g$, as required.
\item We have $V_G\p{K+3-b_2}+V_G\p{K+8-b_4}=V_G\p{10+3-12}+V_G\p{10+8-13}=V_G\p{1}+V_G\p{5}=4+1=5$, as required.
\item We have $V_G\p{K+2-5}=V_G\p{7}=7=12-5$, as required.
\item We have $V_G\p{K-2}=V_G\p{8}=5$, as required.
\item We have $V_G\p{K+4-5}=V_G\p{9}=8=13-5$, as required.
\item In this case, $n_0=1$. We have $V_G\p{5}=1$ and $V_G\p{6}=4$. This means our initial conditions to the recurrence system are $\ic{0}$ for $f$ and $\ic{1}$ for $g$. Furthermore, we see that the Golomb-like system obtained is precisely the one in Proposition~\ref{prop:fg1}, so we obtain those sequences. We now observe that, in those sequences, for $n>1$, $f\p{n}\leq n$, $g\p{n-1}\leq n$, and $g\p{n}\leq n$, meaning this final restriction is satisfied.
\end{itemize}
The above means we have a solution of the form
\[
\begin{cases}
V_G\p{10+5k}=5f\p{k}+1\\
V_G\p{10+5k+1}=5g\p{k}-1\\
V_G\p{10+5k+2}=5k+12\\
V_G\p{10+5k+3}=5\\
V_G\p{10+5k+4}=5k+13,
\end{cases}
\]
beginning at index~$10$. Re-indexing and noting that the pattern actually starts earlier results in the desired solution.
\end{proof}

\begin{pro}\label{prop:v2}
The initial conditions $\ic{3,1,4,2,5,3}$ to the Hofstadter $V$-recurrence produce a solution of the following form for $k\geq2$:
\[
\begin{cases}
V_G\p{5k}=10\\
V_G\p{5k+1}=5k-2\\
V_G\p{5k+2}=5f\p{k+1}+1\\
V_G\p{5k+3}=5g\p{k+1}-1\\
V_G\p{5k+4}=5k-3,
\end{cases}
\]
where $f$ and $g$ are the sequences in Proposition~\ref{prop:fg2}.
\end{pro}
\begin{proof}
Let $K=22$, $b_0=1$, $b_1=-1$, $b_2=17$, $b_4=23$, $a_f=2$, $a_g=8$, and $m=2$. These values satisfy all of the requirements on these parameters. The first $21$ terms of the sequence resulting from the initial conditions $\ic{3,1,4,2,5,3}$ are 
\[
3,1,4,2,5,3,6,4,7,10,8,6,9,7,10,13,6,14,12,10,18,6.
\]
We now check that these satisfy all seven requirements:
\begin{itemize}
\item We have $V_G\p{K+5-b_4}=V_G\p{22+5-23}=V_G\p{4}=2=a_f$, as required.
\item We have $V_G\p{K+6-b_2}=V_G\p{22+6-17}=V_G\p{11}=8=a_g$, as required.
\item We have $V_G\p{K+3-b_2}+V_G\p{K+8-b_4}=V_G\p{22+3-17}+V_G\p{22+8-23}=V_G\p{8}+V_G\p{7}=4+6=10$, as required.
\item We have $V_G\p{K+2-10}=V_G\p{14}=7=17-10$, and $V_G\p{K+2-5}=V_G\p{19}=12=17-5$, as required.
\item We have $V_G\p{K-2}=V_G\p{10}=10$, as required.
\item We have $V_G\p{K+4-10}=V_G\p{16}=13=23-10$, and $V_G\p{K+4-5}=V_G\p{21}=18=23-5$, as required.
\item In this case, $n_0=4$. We have $V_G\p{2}=1$, $V_G\p{7}=6$, $V_G\p{12}=6$, $V_G\p{17}=6$ and $V_G\p{3}=4$, $V_G\p{8}=4$, $V_G\p{13}=9$, $V_G\p{18}=14$. This means our initial conditions to the recurrence system are $\ic{0,1,1,1}$ for $f$ and $\ic{1,1,2,3}$ for $g$. Furthermore, we see that the Golomb-like system obtained is precisely the one in Proposition~\ref{prop:fg2}, so we obtain those sequences. (The initial conditions there are the first three terms of each of these initial conditions, but the fourth terms here equal the fourth terms in those sequences.) We now observe that, in those sequences, for $n>4$, $f\p{n}\leq n$, $g\p{n-1}\leq n$, and $g\p{n}\leq n$, meaning this final restriction is satisfied.
\end{itemize}
The above means we have a solution of the form
\[
\begin{cases}
V_G\p{22+5k}=5f\p{k}+1\\
V_G\p{22+5k+1}=5g\p{k}-1\\
V_G\p{22+5k+2}=5k+17\\
V_G\p{22+5k+3}=10\\
V_G\p{22+5k+4}=5k+23,
\end{cases}
\]
beginning at index~$22$. Re-indexing and noting that the pattern actually starts earlier results in the desired solution.
\end{proof}

\subsection{A Companion to the $V$-Recurrence}
The patterns we observe in the $V$-recurrence all repeat with a period of $5$. The $V$-recurrence, $V\p{n}=V\p{n-V\p{n-1}}+V\p{n-V\p{n-4}}$ prominently features a $1$ and a $4$, which sum to $5$. There is another recurrence, $H\p{n}=H\p{n-H\p{n-2}}+H\p{n-H\p{n-3}}$ with a similar property. In fact, this recurrence seems to be a sort of \emph{companion} to the $V$-recurrence, in that it has a similar families of period-$5$ solutions. Like the $V$-recurrence, $H$ has four fundamentally different families of solutions that eventually consist of interleavings of five constant or linear sequences (see Table~\ref{tab:ich}). More importantly, there is a family of solutions analogous to the solutions to $V$ described in Theorem~\ref{thm:infv}.

\begin{table}
\begin{tabular}{|c|c|}\hline
\textbf{Pattern} & \textbf{Initial Condition}\\\hline
C,C,C,L,L & \[ \ic{5, 3, 0, -1, -1, 5, 0, 1, 4, 2, 5, 3, 10}  \]\\\hline
C,C,L,C,L & \[ \ic{2, 0, 5, 0, 0, 0, 5, 5, 5, 3, 2}  \]\\\hline
C,C,L,L,L & \[ \ic{7, 0, -3, 0, 4, 7, 5, 0, 7, 4, 0, 8, 7, 5, 4, 7, 15, 12, 10}  \]\\\hline
C,C,L,L,L & \[ \ic{6, 1, 0, 3, 3, 0, 6, 4, -1, 3, 6, 0, 12, 4, 3, 6, 16, 14, 9}  \]\\\hline
\end{tabular}
\caption{Patterns and representative initial conditions for each of the four families of period-$5$ solutions to the recurrence $H$. (C=constant, L=linear)}
\label{tab:ich}
\end{table}

\begin{thm}\label{thm:infh}
Let $K$, $b_0$, $b_1$, $b_2$, $b_4$, $a_f$, $a_g$, and $m$ be integers satisfying the following properties:
%\begin{multicols}{2}
\begin{itemize}
\item $b_0\equiv1\pmod{5}$ and $6\leq b_0<K+2$
\item $b_1\equiv4\pmod{5}$ and $9\leq b_1<K+3$
\item $b_2\equiv2\pmod{5}$
\item $b_4\equiv3\pmod{5}$
\item $a_f\equiv 4\pmod{5}$
\item $a_g\equiv 1\pmod{5}$
\item $a_f+a_g>0$
\item $m\geq1$.
\end{itemize}
%\end{multicols}
\noindent
Define the following Golomb-like system:
\[
\begin{cases}
f\p{n}=g\!\pb{n-g\p{n-1}-\frac{b_4+2}{5}}+\frac{b_4-b_2+a_f}{5}\\
g\p{n}=f\!\pb{n-f\p{n}-\frac{b_2-2}{5}}+\frac{b_2-b_4+a_g}{5}.
\end{cases}
\]
%where the first terms in the sequences defined by these recurrences are at index $n_0=-\fl{\frac{K+1}{5}}$. 
Then, there is a solution $H_G$ to the $H$-recurrence that, starting at index $K$, has the form
\[
\begin{cases}
H_G\p{K+5k}=5k+b_0\\
H_G\p{K+5k+1}=5k+b_1\\
H_G\p{K+5k+2}=5f\p{k}+b_2\\
H_G\p{K+5k+3}=5m\\
H_G\p{K+5k+4}=5g\p{k}+b_4
\end{cases}
\]
with any initial condition satisfying the following properties:
\begin{enumerate}
\item\label{it:afh} $H_G\p{K+2-b_0}=a_f$
\item\label{it:agh} $H_G\p{K+4-b_1}=a_g$
\item\label{it:5mh} $H_G\p{K+3-b_0}+V_G\p{K+3-b_1}=5m$
\item\label{it:b0h} For each integer $1\leq i\leq m$, $H_G\p{K-5i}=b_0-5i$
\item\label{it:b1h} For each integer $1\leq i\leq m$, $H_G\p{K+1-5i}=b_1-5i$
%\item\label{it:b3} For each integer $1\leq i\leq m$, $V_G\p{K+3-5i}=5m$
\item\label{it:b3h} $H_G\p{K-2}=5m$
%\item Let $p_1,p_2,\ldots,p_t$ denote the sequence of terms in the initial condition at indices congruent to $K$ mod $5$, and let $q_1,q_2,\ldots,q_t$ denote the sequence of terms in the initial condition at indices congruent to $K+1$ mod $5$ (excluding $V_G\p{1}$ if relevant). The sequences $\frac{p_1-b_0}{5},\frac{p_2-b_0}{5},\ldots,\frac{p_t-b_0}{5}$ and $\frac{q_1-b_1}{5},\frac{q_2-b_1}{5},\ldots,\frac{q_t-b_1}{5}$, when fed to the system for $f$ and $g$
\item\label{it:ich} Let $n_0=\fl{\frac{K+1}{5}}$. The initial conditions
\[
\ic{\frac{H_G\p{K-5n_0+2}-b_0}{5}, \frac{H_G\p{K-5\pb{n_0-1}+2}-b_0}{5}, \frac{H_G\p{K-5\pb{n_0-2+2}}-b_0}{5}, \ldots, \frac{H_G\p{K-3}-b_0}{5}}
\]
for the $f$-sequence and
\[
\ic{\frac{H_G\p{K-5n_0+4}-b_1}{5}, \frac{H_G\p{K-5\pb{n_0-1}+4}-b_1}{5}, \frac{H_G\p{K-5\pb{n_0-2}+4}-b_1}{5}, \ldots, \frac{H_G\p{K-1}-b_1}{5}}
\]
for the $g$-sequence generate sublinear sequences where, for all $n>n_0$, $f\p{n-1}\leq n+\frac{b_2+3}{5}$, $f\p{n}\leq n+\frac{b_2-2}{5}$, and $g\p{n-1}\leq n+\frac{b_4+2}{5}$.%, and $g\p{n}\leq n+\frac{b_4+2}{5}$.
\end{enumerate}
\end{thm}
\begin{proof}
The proof is by induction on the index, with base case provided by the initial condition. Suppose the parameters and initial conditions satisfy all of the listed conditions, and furthermore suppose that the general form of the solution holds through index $n-1$ for some $n\geq K$. We now have five cases to consider:
\begin{description}
\item[$n-K\equiv 0\pmod{5}$:]
In this case, $n=K+5k$ for some $k\geq0$. We have
\[
H_G\p{K+5k}=H_G\p{K+5k-H_G\p{K+5k-2}}+H_G\p{K+5k-H_G\p{K+5k-3}}.
\]
By induction, $H_G\p{K+5k-2}=5m$, as this term falls in the inductive hypothesis or into restriction~\ref{it:b3h} on the initial condition. Similarly, $H_G\p{K+5k+3}=5f\p{k-1}+b_2$, as this term falls in the inductive hypothesis or into restriction~\ref{it:ich}. So, we have
\begin{align*}
H_G\p{K+5k}&=H_G\p{K+5k-5m}+H_G\p{K+5k-5f\p{k-1}-b_2}\\
&=H_G\p{K+5k-5m}+H_G\p{K+5\pb{k-f\p{k-1}}-b_2}.
\end{align*}
We now observe that $H_G\p{K+5k-5m}=5\pb{k-m}+b_0$, as this term falls in the inductive hypothesis or in restriction~\ref{it:b0h} on the initial condition. Also, since $b_2\equiv2\pmod{5}$ and since restriction~\ref{it:ich} guarantees $f\p{k-1}\leq k+\frac{b_2+3}{5}$, we have $H_G\p{K+5\pb{k-f\p{k-1}}-b_2}=5m$. Putting these together yields
\begin{align*}
H_G\p{K+5k}&=5\pb{k-m}+b_0+5m=5k+b_0,
\end{align*}
as required.
%%%%%%%%%%%%%%%
\item[$n-K\equiv 1\pmod{5}$:]
In this case, $n=K+5k+1$ for some $k\geq0$. We have
\[
H_G\p{K+5k+1}=H_G\p{K+5k+1-H_G\p{K+5k-1}}+H_G\p{K+5k+1-H_G\p{K+5k-2}}.
\]
By induction, $H_G\p{K+5k-1}=5g\p{k-1}+b_4$, as this term falls in the inductive hypothesis or in restriction~\ref{it:ich} on the initial condition. Similarly, $H_G\p{K+5k-2}=5m$, as this term either falls in the inductive hypothesis or in restriction~\ref{it:b3h} on the initial condition. So, we have
\begin{align*}
H_G\p{K+5k+1}&=H_G\p{K+5k+1-5g\p{k-1}-b_4}+H_G\p{K+5k+1-5m}\\
&=H_G\p{K+5\pb{k-g\p{k-1}}+1-b_4}+H_G\p{K+5k+1-5m}.
\end{align*}
We now observe that $H_G\p{K+5k+1-5m}=5\pb{k-m}+b_1$, as this term falls in the inductive hypothesis or in restriction~\ref{it:b1h} on the initial condition. Also, since $b_4\equiv3\pmod{5}$ and since restriction~\ref{it:ic} guarantees $g\p{k-1}\leq k+\frac{b_4+2}{5}$, we have $H_G\p{K+5\pb{k-g\p{k-1}}+1-b_4}=5m$. Putting these together yields
\begin{align*}
H_G\p{K+5k+1}&=5m+5\pb{k-m}+b_1=5k+b_1,
\end{align*}
as required.
%%%%%%%%%%%%%%%
\item[$n-K\equiv 2\pmod{5}$:]
In this case, $n=K+5k+2$ for some $k\geq0$. We have
\[
H_G\p{K+5k+2}=H_G\p{K+5k+2-H_G\p{K+5k}}+H_G\p{K+5k+2-H_G\p{K+5k-1}}.
\]
By induction, $H_G\p{K+5k}=5k+b_0$, as this term falls in the inductive hypothesis. Similarly, $H_G\p{K+5k-1}=5g\p{k-1}+b_4$, as this term either falls in the inductive hypothesis or in restriction~\ref{it:ich} on the initial condition. So, we have
\begin{align*}
H_G\p{K+5k+2}&=H_G\p{K+5k+2-5k-b_0}+H_G\p{K+5k+2-5g\p{k-1}-b_4}\\
&=H_G\p{K+2-b_0}+H_G\p{K+5\pb{k-g\p{k-1}}+2-b_4}.
\end{align*}
We now observe that $H_G\p{K+2-b_0}=a_f$ by restriction~\ref{it:afh} on the initial condition. Also, since $b_4\equiv3\pmod{5}$ and since restriction~\ref{it:ich} guarantees $g\p{k-1}\leq k+\frac{b_4+2}{5}$, we have $H_G\p{K+5\pb{k-g\p{k-1}}+2-b_4}=5g\!\pb{k-g\p{k-1}-\frac{b_4+2}{5}}+b_4$. Putting these together yields
\begin{align*}
H_G\p{K+5k+2}&=a_f+5g\!\pb{k-g\p{k-1}-\frac{b_4+2}{5}}+b_4\\
&=5g\!\pb{k-g\p{k-1}-\frac{b_4+2}{5}}+\pb{b_4-b_2+a_f}+b_2\\
&=5f\p{k}+b_2,
\end{align*}
as required.
%%%%%%%%%%%%%%%
\item[$n-K\equiv 3\pmod{5}$:]
In this case, $n=K+5k+3$ for some $k\geq0$. We have
\[
H_G\p{K+5k+3}=H_G\p{K+5k+3-H_G\p{K+5k+1}}+H_G\p{K+5k+3-H_G\p{K+5k}}.
\]
By induction, $H_G\p{K+5k+1}=5k+b_1$, as this term falls in the inductive hypothesis. Similarly, $H_G\p{K+5k}=5k+b_0$, as this term also falls in the inductive hypothesis. So, we have
\begin{align*}
H_G\p{K+5k+3}&=H_G\p{K+5k+3-5k-b_1}+H_G\p{K+5k+3-5k-b_0}\\
&=H_G\p{K+3-b_1}+H_G\p{K+3-b_0}.
\end{align*}
By restriction~\ref{it:5m}, this equals $5m$, as required.
%%%%%%%%%%%%%%%
\item[$n-K\equiv 4\pmod{5}$:]
In this case, $n=K+5k+4$ for some $k\geq0$. We have
\[
H_G\p{K+5k+4}=H_G\p{K+5k+4-H_G\p{K+5k+2}}+H_G\p{K+5k+4-H_G\p{K+5k+1}}.
\]
By induction, $H_G\p{K+5k-1}=5k+b_1$, as this term falls in the inductive hypothesis. Similarly, $H_G\p{K+5k+2}=5f\p{k}+b_2$, as this term falls in the inductive hypothesis. So, we have
\begin{align*}
H_G\p{K+5k+4}&=H_G\p{K+5k+4-5f\p{k}-b_2}+H_G\p{K+5k+4-5k-b_1}\\
&=H_G\p{K+5\pb{k-f\p{k}}+4-b_2}+H_G\p{K+4-b_1}.
\end{align*}
We now observe that $H_G\p{K+4-b_1}=a_g$ by restriction~\ref{it:agh} on the initial condition. Also, since $b_2\equiv2\pmod{5}$ and since restriction~\ref{it:ich} guarantees $f\p{k}\leq k+\frac{b_2-2}{5}$, we have $H_G\p{K+5\pb{k-f\p{k}}+4-b_2}=5f\!\pb{k-f\p{k}-\frac{b_2-2}{5}}+b_0$. Putting these together yields
\begin{align*}
H_G\p{K+5k+4}&=5f\!\pb{k-f\p{k}-\frac{b_2-2}{5}}+b_2+a_g\\
&=5f\!\pb{k-f\p{k}-\frac{b_2-2}{5}}+\pb{b_2-b_4+a_g}+b_4\\
&=5g\p{k}+b_4,
\end{align*}
as required.
\end{description}
\end{proof}

\subsection{Concrete Examples for the $H$-recurrence}
Let us now see a couple of concrete solutions to the $H$-recurrence corresponding to specific settings of the parameters in Theorem~\ref{thm:infh}.
\begin{pro}\label{prop:h1}
The initial conditions $\ic{3,1,4,2}$ to the $H$-recurrence produce a solution of the following form for $k\geq1$:
\[
\begin{cases}
H_G\p{5k}=5\\
H_G\p{5k+1}=5g\p{k}-2\\
H_G\p{5k+2}=5k+1\\
H_G\p{5k+3}=5k+4\\
H_G\p{5k+4}=5f\p{k+1}+2,
\end{cases}
\]
where $f$ and $g$ are the sequences in Proposition~\ref{prop:fg1}.
\end{pro}
\begin{proof}
Let $K=12$, $b_0=11$, $b_1=14$, $b_2=2$, $b_4=-2$, $a_f=4$, $a_g=1$, and $m=1$. These values satisfy all of the requirements on these parameters. The first $11$ terms of the sequence resulting from the initial conditions $\ic{3,1,4,2}$ are $3,1,4,2,5, 3, 6, 9, 7, 5, 3$. We now check that these satisfy all seven requirements:
\begin{itemize}
\item We have $H_G\p{K+2-b_0}=H_G\p{12+2-11}=H_G\p{3}=4=a_f$, as required.
\item We have $H_G\p{K+4-b_1}=H_G\p{12+4-14}=H_G\p{2}=1=a_g$, as required.
\item We have $H_G\p{K+3-b_0}+H_G\p{K+3-b_1}=H_G\p{12+3-11}+H_G\p{12+3-14}=H_G\p{4}+H_G\p{1}=2+3=5$, as required.
\item We have $H_G\p{K-5}=H_G\p{7}=6=11-5$, as required.
\item We have $H_G\p{K+1-5}=H_G\p{8}=9=14-5$, as required.
\item We have $H_G\p{K-2}=H_G\p{10}=5$, as required.
\item In this case, $n_0=2$. We have $H_G\p{4}=2$, $H_G\p{9}=7$ and $H_G\p{6}=3$, $H_G\p{11}=3$. This means our initial conditions to the recurrence system are $\ic{0,1}$ to $f$ and $\ic{1,1}$ to $g$. Furthermore, we see that the Golomb-like system obtained is precisely the one in Proposition~\ref{prop:fg1}, so we obtain those sequences. (The initial conditions there are the first term of our $f$-initial condition, but the other terms here equal the next terms in those sequences.) We now observe that, in those sequences, for $n>2$, $f\p{n-1}\leq n+1$, $f\p{n}\leq n$, and $g\p{n}\leq n$, meaning this final restriction is satisfied.
\end{itemize}
The above means we have a solution of the form
\[
\begin{cases}
H_G\p{12+5k}=5k+11\\
H_G\p{12+5k+1}=5k+14\\
H_G\p{12+5k+2}=5f\p{k}+2\\
H_G\p{12+5k+3}=5\\
H_G\p{12+5k+4}=5g\p{k}-2,
\end{cases}
\]
beginning at index~$12$. Re-indexing and noting that the pattern actually starts earlier results in the desired solution.
\end{proof}

\begin{pro}\label{prop:h2}
The initial conditions $\ic{4,2,5,3,1,4,7,5}$ to the $H$-recurrence produce a solution of the following form for $k\geq3$:
\[
\begin{cases}
H_G\p{5k}=5k-9\\
H_G\p{5k+1}=5k-6\\
H_G\p{5k+2}=5f\p{k}+2\\
H_G\p{5k+3}=10\\
H_G\p{5k+4}=5g\p{k}-2,
\end{cases}
\]
where $f$ and $g$ are the sequences in Proposition~\ref{prop:fg12}.
\end{pro}
\begin{proof}
Let $K=25$, $b_0=16$, $b_1=19$, $b_2=2$, $b_4=-2$, $a_f=9$, $a_g=6$, and $m=2$. These values satisfy all of the requirements on these parameters. The first $24$ terms of the sequence resulting from the initial conditions $\ic{4,2,5,3,1,4,7,5}$ are
\[
4, 2, 5, 3, 1, 4, 7, 5, 3, 6, 9, 7, 10, 8, 6, 9, 12, 10, 13, 11, 14, 12, 10, 13.
\]
We now check that these satisfy all seven requirements:
\begin{itemize}
\item We have $H_G\p{K+2-b_0}=H_G\p{25+2-16}=H_G\p{11}=9=a_f$, as required.
\item We have $H_G\p{K+4-b_1}=H_G\p{25+4-19}=H_G\p{10}=6=a_g$, as required.
\item We have $H_G\p{K+3-b_0}+H_G\p{K+3-b_1}=H_G\p{25+3-16}+H_G\p{25+3-19}=H_G\p{12}+H_G\p{9}=7+3=10$, as required.
\item We have $H_G\p{K-10}=H_G\p{15}=6=16-10$, and $H_G\p{K-5}=H_G\p{20}=11=16-5$, as required.
\item We have $H_G\p{K+1-10}=H_G\p{16}=9=19-10$, and $H_G\p{K+1-5}=H_G\p{21}=14=19-5$, as required.
\item We have $H_G\p{K-2}=H_G\p{23}=10$, as required.
\item In this case, $n_0=5$. We have $H_G\p{2}=2$, $H_G\p{7}=7$, $H_G\p{12}=7$, $H_G\p{17}=12$, $H_G\p{22}=12$ and $H_G\p{4}=3$, $H_G\p{9}=3$, $H_G\p{14}=8$, $H_G\p{19}=13$, $H_G\p{24}=13$. This means our initial conditions to the recurrence system are $\ic{0,1,1,2,2}$ to $f$ and $\ic{1,1,2,3,3}$ to $g$. Furthermore, we see that the Golomb-like system obtained is precisely the one in Proposition~\ref{prop:fg12}, so we obtain those sequences. (The initial conditions there are the first three terms of our initial conditions, but the other terms here equal the next terms in those sequences.) We now observe that, in those sequences, for $n>5$, $f\p{n-1}\leq n+1$, $f\p{n}\leq n$, and $g\p{n}\leq n$, meaning this final restriction is satisfied.
\end{itemize}
The above means we have a solution of the form
\[
\begin{cases}
H_G\p{25+5k}=5k+16\\
H_G\p{25+5k+1}=5k+19\\
H_G\p{25+5k+2}=5f\p{k}+2\\
H_G\p{25+5k+3}=10\\
H_G\p{25+5k+4}=5g\p{k}-2,
\end{cases}
\]
beginning at index~$25$. Re-indexing and noting that the pattern actually starts earlier results in the desired solution.
\end{proof}

\section{Conclusion}

This study sheds light on a new kind of solution while exploring the mysterious nature behind the $V$-recurrence. Especially, $V_{c}(n)$ can be seen as a fascinating example for coexistence of order and chaos in a meta-Fibonacci recurrence although in that case chaotic behaviour brings about termination of corresponding sequence after billions of terms. This perfect mixture of regularity and irregularity reminds that results of Pinn’s study that suggests a physical picture such as terms of random walks in some bizarre surrounding could perhaps help to better understand some of the interesting properties of certain chaotic meta-Fibonacci sequences ~\cite{9}. Indeed, for example, it would be remarkably interesting if the sequences obtained in this study would be helpful to model and calculate the transport of atoms  by altering the site number of the potential in terms of localization and dislocalization properties of the quasi-periodic lattices ~\cite{20}. Some future works could potentially focus on such similar physical application attempts of these curious family of nonlinear recurrences. At the same time, it is known that finding different meta-Fibonacci recurrences with similar behavior is significant and has been an essential key to the substantial progress in terms of new directions in this research area ~\cite{1}. In that direction, this study also provides new connections for two essential nested recurrence relations which are represented by $V$ and $H$ and corresponding results strongly suggest that Hofstadter-Huber generalization is fruitful to discovery for new curious solution sequence families, especially for quasi-periodic relations that this study focuses on.
%\nocite{oreg,schn,pond,smith,marg,hunn,advi,koha,mouse}

%%%%%%%%%%%%%%%%%%%%%%%%%%%%%%%%%%%%%%%%%%%%%%
%%                                          %%
%% Backmatter begins here                   %%
%%                                          %%
%%%%%%%%%%%%%%%%%%%%%%%%%%%%%%%%%%%%%%%%%%%%%%

\begin{backmatter}

\section*{Competing interests}
  The authors declare that they have no competing interests.

\section*{Authors' contributions}
  Altug Alkan is the main contributor of this research. Altug Alkan, Orhan Ozgur Aybar and Zehra Akdeniz analysed the chaotic behaviour of sequence in Section 2, they collaborate in this section. Nathan Fox give significant support in proof related works of this study and Nathan Fox and Altug Alkan collaborate in proof of solutions in section 3. All authors read and approved the final manuscript.

\section*{Acknowledgements}
 Authors would like to thank Robert Israel for his valuable help on Maple-related requirements of this study. Altug Alkan also would like to thank Rémy Sigrist and Giovanni Resta for their valuable computational assistance especially regarding to OEIS contributions such as A309567, A309636, A309650, A309704 and helpful feedbacks about some related sequences of this study.
 
 \section*{Funding}
There is no funding request for this article.

\end{backmatter}

\begin{thebibliography}{99.}%
% Use the following syntax and markup for your references if 
% the subject of your book is from the field 
% "Mathematics, Physics, Statistics, Computer Science"
%
% Contribution 

\bibitem{5} Alkan, A.: On a conjecture about generalized Q-recurrence. Open Mathematics. 16, 1490-1500 (2018)
\bibitem{11} Alkan, A.: On a Generalization of Hofstadters Q-Sequence: A Family of Chaotic Generational Structures. Complexity. 1-8 (2018) 
\bibitem{8} Alkan, A., Fox, N., Aybar O.O.:On Hofstadter Heart Sequences. Complexity. 1-8 (2017)
\bibitem{3} Allouche, J.P., Shallit, J.: A variant of Hofstadter’s sequence and finite automata. J. Aust. Math. Soc.. 93, 1–8 (2012)
\bibitem{2} Balamohan B., Kuznetsov A., Tanny S.: On the behaviour of a variant of Hofstadter’s Q-sequence. J. Integer Sequences. 10, Article 07.7.1. (2007) 
\bibitem{19} Dalton B, Rahman M, Tanny S. Spot-based generations for meta-Fibonacci sequences. Experimental Mathematics 20(2), 129-37. (2011)
\bibitem{14} Fox N.: Discovering Linear-Recurrent Solutions to Hofstadter-Like Recurrences Using Symbolic Computation (2017)
\bibitem{12} Fox, N.: An Exploration Of Nested Recurrences Using Experimental Mathematics. Ph.D. Thesis, Department of Mathematics, Rutgers University (2017)
\bibitem{13} Fox N.: Quasipolynomial Solutions to the Hofstadter Q-Recurrence. Integers. 16 (2016).
\bibitem{16} Golomb, S. W.: Discrete Chaos: Sequences Satisfying Strange Recursions. (1990)
\bibitem{18} Hofstadter, D. \textit{G\:odel, Escher, Bach: an Eternal Golden Braid.} Penguin Books. (1979)
\bibitem{6} Isgur, A.: Solving nested recursions with trees, Ph.D. Thesis, Department of Mathematics, University of Toronto (2012)
\bibitem{1} Isgur, A., Lech, R., Moore, S., Tanny, S., Verberne, Y., Zhang, Y.: Constructing new families of nested recursions with slow solutions. SIAM J. Discrete Math. 30(2), 1128–1147 (2016)
\bibitem{9} Pinn, K.: A Chaotic Cousin of Conways Recursive Sequence. Experimental Mathematics. 9, 55-66 (2000)
\bibitem{10} Pinn, K.: Order and chaos in Hofstadters Q(n) sequence. Complexity. 4, 41-46 (1999)
\bibitem{15} Ruskey, F.: Fibonacci meets Hofstadter. Fibonacci Quart. 49, 227–230 (2011)
\bibitem{7} Sloane, N.J.A.: OEIS Foundation Inc.. The On-Line Encyclopedia of Integer Sequences (2019)
\bibitem{17} Sunohara, M., Tanny, S.: On the solution space of the Golomb recursion., J. Difference Equations and Applications 24.8, 1273-1294. (2018)
\bibitem{4} Tanny S.: An Invitation to Nested Recurrence Relations. Talk given at the 4th biennial Canadian Discrete and Algorithmic Mathematics Conference (Canadam). https://canadam.math.ca/2013/program/slides/Tanny.Steve.pdf (2013)
\bibitem{20} Vignolo, P., Akdeniz Z., and Tosi M. P.: The transmittivity of a Bose{\textendash}Einstein condensate on a lattice: interference from period doubling and the effect of disorder. Journal of Physics B: Atomic, Molecular and Optical Physics. 36, 4535-4546 (2003)

\end{thebibliography}
\end{document}